\documentclass[12pt]{amsart}
\usepackage{fullpage}
\usepackage{amsthm, amstext, amsmath, amssymb,soul}
\usepackage{graphicx}
\usepackage{verbatim}
\usepackage{listings, xcolor}
\usepackage{mathtools}
\usepackage{hyperref}
\usepackage[normalem]{ulem}
\usepackage{lipsum}

\usepackage{soul} %%May delete this package after revising (Wayne).

\title{Integrality and Thurston Rigidity for Bicritical PCF Polynomials}

\author{Heidi Benham}
\address{Department of Mathematics,
         Western Oregon University,
         Monmouth OR 97361}
\email{hbenham17@wou.edu}

\author{Alexander Galarraga}
\address{Department of Mathematics,
         University of Washington,
         Seattle, WA, 98195}
\email{agalar@uw.edu}

\author{Benjamin Hutz}
\address{Department of Mathematics and Statistics,
         Saint Louis University,
         St.~Louis, MO 63103}
\email{benjamin.hutz@slu.edu}

\author{Joey Lupo}
\address{Department of Mathematics and Statistics,
         Amherst College,
         Amherst, MA 01002}
\email{jlupo20@amherst.edu}

\author{Wayne Peng}
\address{Department of Mathematics,
        University of Rochester
        Rochester, NY 14627}
\email{jpeng4@ur.rochester.edu}

\author{Adam Towsley}
\address{School of Mathematical Sciences,
        Rochester Institute of Technology
        Rochester, NY 14623}
\email{adtsma@rit.edu}



\newtheorem{thm}{Theorem}
\newtheorem*{thm*}{Theorem}

\newtheorem{lem}[thm]{Lemma}
\newtheorem{prop}[thm]{Proposition}
\newtheorem*{prop*}{Proposition}
\newtheorem*{conj*}{Conjecture}
\newtheorem{conj}[thm]{Conjecture}
\theoremstyle{definition}
\newtheorem{defn}[thm]{Definition}
\newtheorem{ex}[thm]{Example}

\def\QQ{\mathbb{Q}}

\def\PP{\mathbb{P}}
\def\cP{\mathcal{P}}
\def\cM{\mathcal{M}}

\DeclareMathOperator{\PGL}{PGL}

\thanks{The authors thank the Institute for Computational and Experimental Research in Mathematics, where most of this work was completed during the Summer 2019 Research Experience for Undergraduates. Thanks also to the anonymous referees for their detailed comments which improved this work.}

\begin{document}

\maketitle

\definecolor{codegreen}{rgb}{0,0.6,0}
\definecolor{codegray}{rgb}{0.5,0.5,0.5}
\definecolor{codepurple}{rgb}{0.58,0,0.82}
\definecolor{backcolour}{rgb}{0.95,0.95,0.92}
\lstdefinestyle{mystyle}{
   backgroundcolor=\color{backcolour},
   commentstyle=\color{codegreen},
   keywordstyle=\color{magenta},
   numberstyle=\tiny\color{codegray},
   stringstyle=\color{codepurple},
   basicstyle=\ttfamily\footnotesize,
   breakatwhitespace=false,
   breaklines=true,
   captionpos=b,
   keepspaces=true,
   numbers=left,
   numbersep=5pt,
   showspaces=false,
   showstringspaces=false,
   showtabs=false,
   tabsize=2
}
\lstset{style=mystyle}

\begin{abstract}
    We give an algebraic proof of an important consequence of Thurston rigidity for bicritical PCF polynomials with periodic critical points under certain mild assumptions. The key result is that when the family of bicritical polynomials is parametrized using dynamical Belyi polynomials, the PCF solutions are integral at certain special primes, which we term ``index divisor free primes.'' We prove the existence of index divisor free primes in all but finitely many cases and conjecture the complete list of exceptions. These primes are then used to prove transversality.
\end{abstract}

Let $f(z) \in \QQ(z)$ be a rational function of degree $d\geq2$, considered as an endomorphism of $\PP^1$. Define the $n$\emph{-th iterate} of $f$ recursively as $f^n(z) = f(f^{n-1}(z))$, with $f^0(z) = z$. There is a natural conjugation action on $f$ by $\alpha \in \PGL_2$ given by $f^{\alpha} = \alpha^{-1} \circ f \circ \alpha$. Since the dynamical behavior of $f$ is preserved by this conjugation action, we may consider the set of equivalence classes of degree $d$ rational endomorphisms of $\PP^1$ under $\PGL_2$ conjugation. We denote this moduli space as $\cM_d$, and denote by $\cP_d \subset \cM_d$ the moduli space of degree $d$ polynomials \cite{Silverman, Silverman-transversality}. %section 4.4, introduction
 We denote the conjugacy class in $\cM_d$ represented by the map $f$ as $[f]$.

A \emph{critical point of $f$} is a point with ramification index at least $2$. When the forward orbits of all the critical points are finite, we say the map is \emph{post-critically finite (PCF)}.

One can construct a new moduli space $\cP^{crit}_d$ by marking the critical points of a polynomial. That is, $\cP^{crit}_d$ is equivalence classes of sets of tuples of the form $(f, c_1, \ldots, c_{d-1})$, where $c_1, \ldots, c_{d-1}$ are critical points of the polynomial $f$, each appearing with the appropriate multiplicity. See \cite{Silverman-transversality}, or alternatively \cite[Section 1.5]{PolyPairsBook}, for details. If we require that the critical points are periodic with periods $n_1, \ldots, n_{d-1}$ (a special case of $f$ being PCF) we get a subvariety of $\cP^{crit}_d$. One consequence of Thurston's rigidity theorem \cite{Thurston} is that any two such subvarieties intersect transversely in $\cP^{crit}_d$. The earliest transversality result is perhaps due to Gleason on the family of quadratic polynomials and published by Douady-Hubbard \cite{DH} stating that the roots of $f_c^n(0)$ are simple for $f_c(z) = z^2+c$. For other proofs of transversality for these types of critical orbit relations see for example \cite{BE, LSS}. Favre and Gauthier use this transversality to prove that the PCF parameters equidistribute in the moduli space $\cP_d$ \cite{Favre-Gauthier}. Note that in the cases where Theorem \ref{thm_transversality} holds, our results should allow for a similar equidistribution statement.

Thurston's proof of rigidity relies on complex analytic techniques, which do not generalize nicely. As such, there is interest in arithmetic proofs. Some results include Hutz-Towsley's proof for unicritical polynomials \cite{Hutz-Towsley}, Silverman's proof for degree three polynomials  \cite{Silverman-transversality}, and Epstein's proof for polynomials of degree $p^n$ \cite{Epstein-transversality}. There is also an unpublished work of Levy \cite{Levy} giving an algebraic proof for a certain class of rational maps. Furthermore, a recent preprint of Ji and Xie \cite{JX} gives a new method of proof that does not rely on Teichm\"uller theory for a version of rigidity due to McMullen, which is based on Thurston's rigidity.

We give an algebraic proof of Thurston rigidity for the case of bicritical polynomials, that is, polynomials with two affine critical points. We restrict our attention to $\mathcal{B}_{d,k}^{crit}$, which we define to be the subset of $\cP^{crit}_d$ of equivalence classes of the form $(f, c_1, \ldots, c_1, c_2, \ldots, c_2)$, where $c_1$ occurs $k$ times.

\begin{thm} \label{thm_transversality}
    For integers $n, m \geq 1$ let
    \begin{align*}
        C_{d,0,n} &= \{ (f, c_1, \ldots, c_1, c_2, \ldots, c_2) \in B^{crit}_{d,k} \;\vert\; c_1 \text{ is periodic with } f^n(c_1) = c_1\} \\
        C_{d,1,m} &= \{ (f, c_1, \ldots, c_1, c_2, \ldots, c_2) \in B^{crit}_{d,k} \;\vert\; c_2 \text{ is periodic with } f^m(c_2) = c_2\}.
    \end{align*}
    With the exception of finitely many pairs $(d,k)$, the curves $C_{d,0,m}$ and $C_{d,1,n}$ intersect transversely.
\end{thm}

Unfortunately, the methods used do not allow us to explicitly describe the finite exceptions $(d,k)$. We conjecture a list of all exceptions in Conjecture \ref{conj:goodprime}.

Our method of proof follows the general plan of attack used by both Silverman and Epstein: find a prime at which we can reduce and demonstrate that the Jacobian does not vanish at a point of intersection. In both Epstein and Silverman, finding a prime such that the points of intersection are integral is simple: take the prime dividing the degree. Integrality in our case turns out to be more complicated and requires careful selection of the prime. %alex: changes made here as I realized some of what was written incorrectly characterized Epstein's work
Overall, the details of our proof are of a similar nature to Epstein's proof in that we carefully analyze the valuation of the forward orbit of the critical points to show that the critical points must be $p$-adically integral. Finally, we calculate a specific Jacobian modulo $p$ to show it is non-zero.

When compared to previous results, our result is more general in that it holds for polynomials of any degree. Our result is also interesting in that we reduce modulo a prime which is (almost always) a prime of bad reduction and the prime varies depending on the map being considered. However, our results are limited in that we only consider bicritical polynomials, as we rely on dynamical Belyi polynomials, and Belyi polynomials always have two affine critical points. Both Silverman and Epstein work with monic centered form, and as Epstein notes in \cite{Epstein-transversality}, monic centered form is not sufficient in the non-prime power degree case as the PCF solutions are not necessarily $p$-adically integral. Our use of dynamical Belyi polynomials to parameterize bicritical polynomials avoids this issue, as we prove the following key proposition.

\begin{prop}\label{prop:main}
Let $f_{a,c}(z) = a \mathcal{B}_{d,k}(z) + c$, where $\mathcal{B}_{d,k}(z)$ is a normalized Belyi polynomial. With finitely many exceptions $(d,k)$, there exists a prime $p$ such that if $f_{\alpha,\beta}(z)$ is PCF with periodic critical points, then
\begin{align*}
    v_p(\alpha) = 0 \quad v_p(\beta) \geq 0
\end{align*}
where $v_p$ is the normalized $p$-adic valuation.
\end{prop}
\noindent Proposition \ref{prop:main} allows for the reduction of PCF polynomials in Belyi normal form with periodic critical points modulo a nice prime $p$.

The prime appearing in Proposition ~\ref{prop:main} is from a special class of primes which we have not seen in the literature, which we term \emph{index divisor free primes}, or IDF primes. Given a tuple $(d,k)$ where $k \leq \left\lceil \frac{d-2}{2} \right\rceil$, we define an IDF prime for $(d,k)$ to be a prime $p$ such that
\begin{itemize}
    \item $p$ is greater than $k$.
    \item $p$ divides $d-r$ for some $r$ less than or equal to $k$.
    \item $r \nmid v_p(d-r)$, where $v_p$ is the normalized $p$-adic valuation.
\end{itemize}
We prove that except for finitely many explicitly computable tuples $(d,k)$, there exists an IDF prime for $(d,k)$. Our results inspire the following conjecture:

\begin{conj}\label{conj:goodprime}
Except when $(d,k)$ equals $(27, 3)$, there exists an index divisor free prime for $(d,k)$
\end{conj}

While IDF primes themselves appear to be unstudied, they are closely related to prime powers in products of consecutive integers, of which there is an extensive literature \cite{ProductNeverPower,FiniteSolutions,GreatestPrime}. The following conjecture, due to Erd\"os and Selfridge \cite{ProductNeverPower}, implies our conjecture when $k$ is greater than or equal to 5 and not a prime.

\begin{conj}[Erd\"os and Selfridge] \label{conj:erdos}
If $k \geq 4$ and $n + k \geq p^{(k)}$, where $p^{(k)}$ is the smallest prime greater than or equal to $k$, then there is a prime greater than $k$ which divides $(n+1)\cdots(n+k)$ to the first power.
\end{conj}

Thus, the existence of an IDF prime can therefore be viewed as a weakening of Conjecture ~\ref{conj:erdos} in most cases.

Finally, as the strongest hypothesis of the main theorem is that the polynomial is bicritical, one might wonder what happens if we try and replace bicritical with $n$-critical for some $n$ greater than 2. As Section \ref{n-critical-points} shows, we can construct a generalization of the dynamical Belyi polynomials for any fixed number of critical points $n$, however, we produce a counterexample to show that the same method of proof is fruitless. Consequently, a general algebraic proof of transversality will require a new method.

This article is organized as follows. Sections 1 and 3 prove the essential ingredients of the proof of Theorem \ref{thm_transversality} under the assumption of the existence of an index-divisor free prime. Section 2 investigates the existence of IDF primes, showing that in all but finitely many cases an IDF prime exists. Finally, Section 4 shows that the natural extension to more than two critical points fails to provide similar results.

\section{Integrality}\label{integrality}

Following the outline for proving rigidity given in Silverman \cite{Silverman-transversality}, we first prove that when the space of bicritical polynomials is appropriately parametrized, every PCF solution is $p$-adically integral. We first parametrize the space of bicritical polynomials using dynamical Belyi polynomials as done in Tobin \cite{Bella}. Throughout let $K$ be a field of characteristic $0$ and $\overline{K}$ an algebraic closure of $K$.

First let us recall a normal form of a single-cycle Belyi map by \cite[Proposition 3.1]{Dynamical}
\begin{equation}\label{eq:Bdk}
    \mathcal{B}_{d,k}(z) := \sum_{i=0}^{k} \frac{(-1)^{k-i}}{(k - i)!i!}  \left(\,\prod_{j = 0, j \neq i}^k d - j\right)z^{d - i}.
\end{equation}

We state some results from Tobin \cite{Bella} which we will use without proof.
\begin{lem}[\textup{\cite[Proposition 4.0.2]{Bella}}] \label{Bella}
Let $g \in K[z]$ be a bicritical polynomial of degree $d \geq 3$. Then $g$ is conjugate to a map $f_{a, c}(z) = a\mathcal{B}_{d,k}(z) + c$, where $\mathcal{B}_{d,k}(z)$ a single-cycle Belyi map and $a, c \in \overline{K}$. Moreover, $f_{a,c}(z)$ has affine critical points $\{0,1\}$.
\end{lem}
Note that the proof of \cite[Proposition 4.0.2]{Bella} implies that the ramification index of $0$ with respect to $f_{a,c}(z) = a\mathcal{B}_{d,k} + c$ is $d-k$, while the ramification index of $1$ is $k+1$ and that 0 is a critical point implies a ramification index of at least 2, so that $k \leq d- 2$. Similarly, the ramification index of $1$ must be at least 2, so that $1\leq k \leq d-2$.
\begin{lem}[\textup{\cite[Proposition 4.0.6]{Bella}}] \label{conjugacy}
Let $f_0 \neq f_1 \in K[z]$ with $f_0 (z) = a_0 B_{d,k_0} + c_0$  and $f_1 (z) = a_1 B _{d,k_1} + c_1$ . The
polynomials $f_0$ and $f_1$ are conjugate if and only if $k_0 + k_1 = d - 1$, $a_0 = a_1$, and $c_1 = 1 - a_0 - c_0 $.
\end{lem}

Combining Lemma \ref{conjugacy} with the inequality $1 \leq k \leq d-2$, we note that $f_{a,c}(z) = a\mathcal{B}_{d,k}(z) +c$ can be chosen such that $1 \leq k \leq \left\lceil \frac{d-2}{2} \right\rceil$.

\begin{lem} \label{bound on k}
Let $g \in K[z]$ be a bicritical polynomial of degree $d \geq 3$. Then $g$ is conjugate to a map $f_{a, c}(z) = a\mathcal{B}_{d,k}(z) + c$, where $1 \leq k \leq \left\lceil \frac{d-2}{2} \right\rceil$.
\end{lem}

Hence, by making an appropriate change of variables, we may assume that our bicritical polynomial has the form $f_{a,c}(z)$ with marked critical points 0 and 1.

Now we write the equations in terms of $a$ and $c$ defining when the two critical points are periodic, i.e., when $f$ is a PCF bicritical polynomial. Define the polynomials
\begin{align*}
    F_n(a,c) = f^n_{a,c}(0) \quad G_m(a,c) = f^m_{a,c}(1)-1.
\end{align*}
The solutions $(\alpha,\beta)$ to
\begin{align*}
F_n(a,c) = G_m(a,c) = 0
\end{align*}
and are exactly the pairs $(\alpha, \beta)$ such that $f_{\alpha, \beta}(z)$ is post-critically finite with 0 and 1 being periodic with periods $n$ and $m$ respectively. Thus, we want to prove that such solutions $\alpha$ and $\beta$ are $p$-adically integral for some prime $p$. For completeness, we restate the definition given in the introduction for the class of primes we consider.

\begin{defn}
A prime $p$ is a \emph{index divisor free} prime for $(d,k)$ when all the following conditions hold.
\begin{itemize}
    \item $p$ is greater than $k$.
    \item $p$ divides $d-r$ for some $r$ less than or equal to $k$. As $p > k$, this $r$ is unique.
    \item For the unique $r$ above, $r \nmid v_p(d-r)$, where  $v_p$ is the normalized $p$-adic valuation.
\end{itemize}
\end{defn}
\noindent Note that we do not require that $p$ is a prime of good reduction for $f_{a,c}(z) = a\mathcal{B}_{d,k}(z) + c$. Also, note that $r$ can not equal 1, as 1 divides every number, violating the third condition. The third condition reads as the index $r$ does not divide the power of $p$ dividing $d-r$, hence the name ``index divisor free.'' We abbreviate ``index divisor free'' as \emph{IDF}.
An IDF prime ideal is defined in a similar fashion. When working in extensions of $\mathbb{Q}$, we use IDF prime ideals instead of primes.

Let us examine how index divisor free primes for $(d,k)$ relate to $\mathcal{B}_{d,k}(z)$. We label the coefficients of $\mathcal{B}_{d,k}(z)$ as $b_0, b_1, \ldots , b_k$. Note that the denominator of $b_i$ is $(k-i)! i!$, which contains no powers of primes greater than $k$ as both $k-i$ and $i$ are at most $k$, so that all $b_i$ are $p$-adically integral. Also note that since $p$ divides $d-r$ for some $0 \leq r \leq k$, $p$ must divide every $b_i$ except for $b_r$, as $d-r$ will occur in the product $\prod_{j=0, j \neq i}^k d - j$ except when $r = i$. Moreover, we must have that $v_p(b_i) = v_p(d-r)$ except for $v_p(b_r)$, which is 0.

To further motivate the above definition, note that first two conditions together guarantee that the reduction of $f_{a,c}(z)$ modulo an IDF prime is a monomial. This allows control of the forward orbit of the critical points modulo $p$. The third condition arises as a technicality in our proof that the post-critically finite solutions $(\alpha, \beta)$ are $p$-adically integral.

We now investigate the valuation with respect to an IDF prime of the image of a point under $f_{a,c}(z)$, showing that the hypothesis of IDF primes give some control of the valuation of the image.
\begin{lem}\label{lem:valuation_of_image}
Let p be an IDF prime for $(d,k)$, and let $f_{a,c}(z) = a\mathcal{B}_{d,k}(z) + c$. Let $\alpha$ and $\beta$ be algebraic over $\mathbb{Q}$, and let $x \in \mathbb{Q}(\alpha, \beta)$. Let $\mathfrak{p}$ be a prime ideal of $\mathbb{Q}(\alpha, \beta)$ lying above $p$. Then, if $v_\mathfrak{p}(x)$ is zero, we have that
\begin{equation}\label{eq:mineq0}
    v_{\mathfrak{p}}(f_{\alpha,\beta}(x)) \geq \min \{ v_{\mathfrak{p}}(\alpha), v_{\mathfrak{p}}(\beta) \}
\end{equation}
while if $v_\mathfrak{p}(x)$ is negative, we have that
\begin{equation}\label{eq:minleq0}
    v_\mathfrak{p}(f_{\alpha, \beta}(x)) \geq \min \{ v_\mathfrak{p}(\alpha) + v_\mathfrak{p}(d-r) + d v_\mathfrak{p}(x), v_\mathfrak{p}(\alpha) + (d-r) v_\mathfrak{p}(x), v_\mathfrak{p}(\beta) \}.
\end{equation}
Both of the inequalities \eqref{eq:mineq0} and \eqref{eq:minleq0} are equalities if and only if the minimum is unique. Further, $v_\mathfrak{p}(\alpha) + v_\mathfrak{p}(d-r) + d v_\mathfrak{p}(x)$ does not equal $v_\mathfrak{p}(\alpha) + (d-r) v_\mathfrak{p}(x)$.
\end{lem}
\begin{proof}
We assume for simplicity that $\alpha, \beta \in \mathbb{Q}$ and thus that $\mathfrak{p}$ is $p$, as the proof is not substantially different when $\alpha, \beta \not \in \mathbb{Q}$. Let $x \in \mathbb{Q}(\alpha, \beta)$. Consider $v_p(f_{\alpha,\beta}(x))$. Using the labels $b_0, \ldots , b_k$ for the coefficients of $\mathcal{B}_{d,k}(z)$ as above, we have
\begin{align*}
    v_p(f_{\alpha, \beta}(x)) &= v_p(\alpha(\mathcal{B}_{d,k}(x)) + \beta)
    = v_p(\alpha(b_0x^d + b_1 x^{d-1} + \ldots + b_k x^{d-k})+\beta)\\
    &\geq \min \{v_p(\alpha b_0 x^d), v_p(\alpha b_1 x^{d-1}), \ldots , v_p(\alpha b_k x^{d-k}), v_p(\beta) \}\\
    &= \min \{ v_p(\alpha) + v_p(b_0) + d v_p(x), \ldots, v_p(\alpha) + v_p(b_k) + (d-k) v_p(x), v_p(\beta) \},
\end{align*}
where equality occurs if there is a unique minimum. Since $p$ is an IDF prime, fix $r$ so that $p \mid (d-r)$. We know that $v_p(b_i) = v_p(d-r)$ for all $i$, with the exception that $v_p(b_r) = 0$. Substituting into the set we are minimizing, we find that
\begin{multline*}
   v_p(f_{\alpha, \beta}(x)) \geq \min \{ v_p(\alpha) + v_p(d-r) + d v_p(x), \ldots, v_p(\alpha) + (d-r) v_p(x), \ldots, \\
   v_p(\alpha) + v_p(d-r) + (d-k) v_p(x), v_p(\beta) \}.
\end{multline*}
If $v_p(x)$ is zero, then we have that
\begin{align*}
    v_p(f_{\alpha,\beta}(x)) \geq \min \{ v_p(\alpha) + v_p(d-r), \ldots, v_p(\alpha), \ldots, v_p(\alpha) + v_p(d-r), v_p(\beta) \}
\end{align*}
which gives Equation (\ref{eq:mineq0}) as $v_p(d-r) > 0$.

If $v_p(x)$ is negative, then $v_p(\alpha) + v_p(d-r) + dv_p(x)$ is less than $v_p(\alpha) + v_p(d-r) + (d-1) v_p(x)$, and less than $v_p(\alpha) + v_p(d-r) + (d-2) v_p(x)$, etc. So in this case we can ignore all but three terms when we minimize, giving Equation (\ref{eq:minleq0}).
As we are concerned with when the minimum is unique, we note here that $v_p(\alpha) + v_p(d-r) + d v_p(x)$ cannot equal $v_p(\alpha) + (d-r) v_p(x)$, as then we would have that
\begin{align*}
    v_p(d-r) = -rv_p(x)
\end{align*}
which is not possible since $p$ is an IDF prime so that $r$ does not divide $v_p(d-r)$.
\end{proof}

We will need the following technical lemma to handle Case 4.iii in the proof of the main result in this section (Proposition \ref{prop:pintegral}).
\begin{lem}\label{lem:fxy}
    Assume $v_p(\beta) < 0 < v_p(\alpha)$ and $\min\{v_p(\alpha) + v_p(d-r) + dv_p(\beta), v_p(\alpha) + (d-r)v_p(\beta)\} = v_p(\beta)$. For any integer $n\geq0$, we can express
    \begin{equation}\label{eq:fxy}
    f^n_{\alpha, \beta}(X+Y) = f^{n}_{\alpha, \beta}(X)+h_n(X,Y)
    \end{equation}
    for some polynomial $h_n(X,Y)$.
    Further, if $x, y \in \mathbb{Q}(\alpha, \beta)$, with $v_p(x) \geq v_p(\beta)$ and $v_p(y) \geq v_p(\alpha)$, then $v_p(h_n(x, y))$ is greater than or equal to $v_p(\alpha)$ and $v_p(f^{n}_{\alpha,\beta}(x))$ is greater than or equal to $v_p(\beta)$.
\end{lem}
\begin{proof}
    Begin with the case $n=0$. Then, $f^0_{\alpha,\beta}(X+Y) = X+Y = f_{\alpha,\beta}^0(X) + Y$, and we can verify all the claims immediately. Thus, we can assume that $n > 0$.
    Observe that
    \begin{align*}
    f_{\alpha, \beta}(X+Y)&= \alpha(b_{d}(X+Y)^d+\cdots+b_{d-k}(X+Y)^{d-k})+\beta\\
    &=\alpha(b_dX^d+\cdots+b_{d-k}X^{d-k})+\beta+\alpha \sum_{j=d-k}^d b_j\sum_{i=1}^j {j \choose i} X^{j-i}Y^i\\
    &=f_{\alpha,\beta}(X)+h_1(X,Y).
    \end{align*}
    Letting
    \[
    h_n(X,Y)=\alpha\sum_{j=k}^d b_j\sum_{i=1}^j{j\choose i} (f_{\alpha,\beta}^{n-1}(X))^{j-i} (h_{n-1}(X, Y)^i),
    \]
    it is easy to check Equation~\eqref{eq:fxy}.

    Let us now focus on the second part of this lemma. We begin by proving that $v_p(f^{n}_{\alpha, \beta}(x))$ is greater than or equal to $v_p(\beta)$ for all $n$ by induction. As $f^0_{\alpha, \beta}(x)$ equals $x$, the base case is true by assumption. Let us consider the general case. We can start by applying the assumptions, and then use the same logic that lead to Equation ~\eqref{eq:minleq0}:
    \begin{align*}
        v_p(f^n_{\alpha, \beta}(x)) &\geq \min \{v_p(\alpha) + v_p(b_d) + dv_p(f^{n-1}_{\alpha, \beta}(x)), \ldots,  v_p(\alpha) + v_p(b_{d-k}) + \\ & \hspace{7cm} (d-k)v_p(f^{n-1}_{\alpha,\beta}(x)), v_p(\beta)  \} \\
        &\geq \min \{v_p(\alpha) + v_p(b_d) + dv_p(\beta), \ldots, v_p(\alpha) + v_p(b_{d-k}) + (d-k)v_p(\beta), v_p(\beta)  \}\\
        &\geq \min \{v_p(\alpha) + v_p(d-r) + dv_p(\beta), v_p(\alpha) + (d-r)v_p(\beta), v_p(\beta)  \} = v_p(\beta),
    \end{align*}
    and, hence, we have shown the desired result.

    We proceed in a similar manner to prove that $v_p(h_n(x,y))$ is greater than or equal to $v_p(\alpha)$ for all $n$. The base case $n=1$ is almost identical to the general case, so we only consider the general case. Using the definition of $h_n(x,y)$, we find that
    \begin{align*}
    v_p(h_n(x, y)) &\geq \min \{v_p(\alpha) + v_p(b_j) + v_p\left( {j \choose i} \right) + (j-i) v_p(f^{n-1}_{\alpha, \beta}(x)) + iv(h_{n-1}(x, y)) \\
    &\hspace{6cm} \mid j = d-k,\ldots,d,\ \text{and}\ i=1,\ldots,j\}\\
    &\geq\min\{v_p(b_j)+v_p\left({j \choose i}\right) + (j-i)v_p(\beta) + (i+1)v_p(\alpha)\\
    &\hspace{6cm} \mid j=d-k,\ldots,d,\ \text{and}\ i=1,\ldots,j\}.
    \end{align*}
    We know that $v_p(b_j)$ equals $v_p(d-r)$ except when $j = d-r$, in which case $v_p(b_{d-r})$ equals zero. Similarly, $v_p\left({j \choose i}\right)$ equals either $v_p(d-r)$ or 0. Using these facts and the assumptions that $v_p(\beta)$ is negative and $v_p(\alpha)$ is positive, we find that minimum must occur when $(j,i)$ equals either $(d,1)$ or $(d-r,d-r)$ so we have
    \begin{equation}\label{eq:h}
    v_p(h_{n}(x, y)) \geq\min\{ v_p(d-r) + (d-1)v_p(\beta) + 2v_p(\alpha), (d-r+1)v_p(\alpha)\},
    \end{equation}
    where we have dropped $v_p\left( {j \choose i} \right)$ as it is positive.
    The second term of the above minimum is clearly greater than $v_p(\alpha)$. The assumption
    \[
    \min\{v_p(\alpha)+v_p(d-r)+dv_p(\beta),v_p(\alpha)+(d-r)v_p(\beta)\}=v_p(\beta)
    \]
    implies
    \[
    \min\{v_p(\alpha)+v_p(d-r)+(d-1)v_p(\beta),v_p(\alpha)+(d-r-1)v_p(\beta)\}=0.
    \]
    It follows that the first term of the minimum in Equation~\eqref{eq:h} is also greater than $v_p(\alpha)$, which is our desired result.
\end{proof}

We are now ready to state the main result of this section (Proposition \ref{prop:pintegral}), which is a restatement of part of Proposition \ref{prop:main}. The remaining part of Proposition \ref{prop:main}, the existence of IDF primes, is treated in Section \ref{idfprimes}.

\begin{prop} \label{prop:pintegral}
    Let $p$ be an IDF prime for $(d,k)$ and let $f_{a,c}(z) = a \mathcal{B}_{d,k}(z) + c$. Then the solutions $(\alpha, \beta)$ to $F_n(a,c) = G_m(a,c) = 0$ are $p$-adically integral.
\end{prop}
\begin{proof}
Our general approach will be to show that if $v_p(\alpha)$ or $v_p(\beta)$ is negative, the $p$-adic valuation of the forward orbit of either 0 or 1 is negative and bounded from above.

Since the proof for an algebraic extension is not substantially different from the rational case, we will assume $\alpha$ and $\beta$ are rational for simplicity.

In order to prove the proposition, we divide into the following cases.
\begin{align*}
        \text{\textbf{Case 1.} } &v_p(\alpha) < 0,\;  v_p(\beta) < 0 {;}\\
        \text{\textbf{Case 2. }} &v_p(\alpha) < 0 \leq v_p(\beta){;} \\
        \text{\textbf{Case 3. }} &v_p(\beta) < 0,\; v_p(\alpha) = 0 {;}\\
        \text{\textbf{Case 4. }} &v_p(\beta) < 0 < v_p(\alpha){.}
\end{align*}

\textbf{Case 1.} We show by induction that $v_p(f^n_{\alpha, \beta}(0))$ tends towards $-\infty$ by showing that $v_p(f^{n+1}_{\alpha, \beta} (0))$ is strictly less than $v_p(f^n_{\alpha, \beta}(0))$ so that $0$ cannot be periodic. For the base case, $f_{\alpha, \beta}(0) = \beta$ and hence $v_p(f_{\alpha, \beta}(0)) = v_p(\beta) < v_p(0)$. Now consider $v_p(f^2_{\alpha,\beta}(0))$. In this case we have assumed $v_p(\beta) < 0$, so we can use Equation ~\eqref{eq:minleq0} to find that
\begin{align*}
    v_p(f^2_{\alpha,\beta}(0)) = v_p(f_{\alpha, \beta}(\beta)) \geq \min \{ v_p(\alpha) + v_p(d-r) + d v_p(\beta), v_p(\alpha) + (d-r) v_p(\beta), v_p(\beta) \}.
\end{align*}
The above inequality is an equality if there is a unique minimum. We have shown in Lemma \ref{lem:valuation_of_image} that $v_p(\alpha) + v_p(d-r) + d v_p(\beta)$ is not equal to $v_p(\alpha) + (d-r) v_p(\beta)$.

We show that $v_p(\alpha) + (d-r) v_p(\beta)$ is less than $v_p(\beta)$.
If $v_p(\alpha) + (d-r) v_p(\beta)$ were greater than or equal to $v_p(\beta)$, then we would have $v_p(\alpha)$ is greater than or equal to $(1-d+r)v_p(\beta)$. As $r \leq k < \left\lceil \frac{d-2}{2} \right\rceil$, $(1-d+r)v_p(\beta)$ is positive and hence $v_p(\alpha)$ would also be positive, a contradiction. Hence, the minimum is unique, and
\begin{align*}
    v_p(f_{\alpha, \beta}(\beta)) =  \min \{ v_p(\alpha) + v_p(d-r) + d v_p(\beta), v_p(\alpha) + (d-r) v_p(\beta)\}.
\end{align*}

Thus, $v_p(f^2_{\alpha,\beta}(0))$ equals $v_p(\alpha) + v_p(d-r) + d v_p(\beta)$ or $v_p(\alpha) + (d-r) v_p(\beta)$ and it must be less than $v_p(\beta)$. As $v_p(\beta)$ equals $v_p(f_{\alpha,\beta}(0))$, we thus have that $v_p(f^2_{\alpha,\beta}(0))$ is less than $v_p(f_{\alpha,\beta}(0))$. Application of Equation ~\eqref{eq:minleq0} in the inductive step combined with the fact that $v_p(\alpha) + v_p(d-r) + d v_p(\beta)$ and $v_p(\alpha) + (d-r)v_p(\beta)$ are less than $v_p(\beta)$ show that in this case $v_p(f^{n+1}_{\alpha, \beta} (0))$ is less than $v_p(f^n_{\alpha, \beta}(0))$.

As $v_p(f^n_{\alpha, \beta}(0))$ is always negative, $0$ is not periodic and $f_{\alpha, \beta}$ is not post-critically finite, so that $(\alpha, \beta)$ is not a solution to $F_n(a,c) = G_m(a,c) = 0$.

\textbf{Case 2.} We show by induction that $v_p(f^n_{\alpha, \beta}(1))$ tends to $-\infty$ by showing that $v_p(f^{n+1}_{\alpha, \beta}(1))$ is strictly less than $v_p(f^{n}_{\alpha, \beta}(1))$. In particular, $1$ is not periodic. To start the base case, we can use Equation ~\eqref{eq:mineq0} and the assumptions that $v_p(\alpha)$ is negative and $v_p(\beta)$ is non-negative
\begin{align*}
    v_p(f_{\alpha, \beta}(1)) = \min \{v_p(\alpha), v_p(\beta) \} = v_p(\alpha)
\end{align*}
with equality since $v_p(\alpha)$ is less than $v_p(\beta)$. Now consider $v_p(f^2_{\alpha, \beta}(1))$. As $v_p(f_{\alpha, \beta}(1))$ is negative, we can use Equation ~\eqref{eq:minleq0}
\begin{align*}
    v_p(f^2_{\alpha, \beta}(1)) &\geq \min \{ v_p(\alpha) + v_p(d-r) + d v_p(\alpha), v_p(\alpha) + (d-r) v_p(\alpha), v_p(\beta) \} \\
    &= \min \{ v_p(d-r) + (d+1) v_p(\alpha), (d-r+1) v_p(\alpha)\}.
\end{align*}
Lemma \ref{lem:valuation_of_image} shows that the minimum is unique, and, hence, the inequality becomes an equality. Thus, we have that $v_p(f^2_{\alpha, \beta}(1))$ is strictly less than $v_p(f_{\alpha, \beta}(1))$. By repeated use of Equation ~\eqref{eq:minleq0}, induction on $n$ shows that $v_p(f^{n+1}_{\alpha, \beta}(1))$ is strictly less than $v_p(f^{n}_{\alpha, \beta}(1))$ for all $n$.

\textbf{Case 3.} We show by induction that $v_p(f^n_{\alpha, \beta}(1))$ tends to $-\infty$, so that $1$ is not periodic. For the base case, we have
\begin{align*}
    v_p(f_{\alpha, \beta}(1)) = \min \{v_p(\alpha), v_p(\beta) \} = v_p(\beta)
\end{align*}
with equality as $v_p(\beta)$ is less than $v_p(\alpha)$. Using Equation ~\eqref{eq:minleq0}, we find that
\begin{align*}
    v_p(f^2_{\alpha, \beta}(1)) &\geq \min \{ v_p(\alpha) + v_p(d-r) + d v_p(\beta), v_p(\alpha) + (d-r) v_p(\beta), v_p(\beta) \} \\
    &= \min \{ v_p(d-r) + dv_p(\beta), (d-r)v_p(\beta)\}
\end{align*}
and Lemma ~\ref{lem:valuation_of_image} shows that the minimum is unique, and, hence, the inequalities become equality. Thus, we have that $v_p(f^2_{\alpha, \beta}(1)$ is strictly less than $v_p(f_{\alpha, \beta}(1))$. By repeated use of Equation ~\eqref{eq:minleq0}, induction on $n$ shows that $v_p(f^{n+1}_{\alpha, \beta}(1))$ is strictly less than $v_p(f^{n}_{\alpha, \beta}(1))$ for all $n$.

\textbf{Case 4.} To deal with this case, which is by far the most difficult, we further divide into 3 subcases.
\begin{align*}
    \text{\textbf{Case 4.i. }} &\min\{v_p(\alpha) + v_p(d-r) + dv_p(\beta), v_p(\alpha) + (d-r)v_p(\beta) \} < v_p(\beta){;} \\
    \text{\textbf{Case 4.ii. }} &\min\{v_p(\alpha) + v_p(d-r) + dv_p(\beta), v_p(\alpha) + (d-r)v_p(\beta) \} > v_p(\beta){;} \\
    \text{\textbf{Case 4.iii. }} &\min\{v_p(\alpha) + v_p(d-r) + dv_p(\beta), v_p(\alpha) + (d-r)v_p(\beta)\} = v_p(\beta){.}
\end{align*}
Cases 4.i. and 4.ii. are easy to deal with. We apply similar arguments to those used for Case 1 to show that in Case 4.i., $v_p(f^n_{\alpha, \beta}(0))$ tends to $-\infty$. As $f_{\alpha, \beta}(0)$ equals $\beta$, we have that $v_p(f_{\alpha, \beta}(0))$ is strictly less than $v_p(f^0_{\alpha, \beta}(0))$, and thus the base case is satisfied. Using Equation (\ref{eq:minleq0}) and the Case 4.i. assumption, we have that
\begin{align*}
    v_p(f_{\alpha, \beta}(\beta)) =  \min \{ v_p(\alpha) + v_p(d-r) + d v_p(\beta), v_p(\alpha) + (d-r) v_p(\beta)\}
\end{align*}
with equality by Lemma (\ref{lem:valuation_of_image}). Thus, $v_p(f^2_{\alpha, \beta}(0))$ is less than $v_p(\beta)$, and induction combined with Equation (\ref{eq:minleq0}) shows that $v_p(f^{n+1}_{\alpha, \beta}(0))$ is strictly less than $v_p(f^n_{\alpha, \beta}(0))$, giving the desired result.

In Case 4.ii., we can show that $v_p(f^n_{\alpha, \beta}(0))$ equals $v_p(\beta)$ for all $n \geq 1$. For the base case $n=1$, we have that $f_{\alpha, \beta}(0)$ equals $\beta$, and thus the base case holds. Then, the inductive hypothesis shows that the assumption for Equation (\ref{eq:minleq0}) is satisfied, and hence we have that
\begin{align*}
    v_p(f^{n+1}_{\alpha, \beta}(0)) &\geq \min \{ v_p(\alpha) + v_p(d-r) + d v_p(f^n_{\alpha,\beta}(0)), v_p(\alpha) + (d-r) v_p(f^n_{\alpha,\beta}(0)), v_p(\beta) \} \\
    & \geq \min \{ v_p(\alpha) + v_p(d-r) + d v_p(\beta), v_p(\alpha) + (d-r) v_p(\beta), v_p(\beta) \}
\end{align*}
By the assumption for Case 4.ii, this minimum is $v_p(\beta)$ and is unique, and thus we have that $v_p(f^{n+1}_{\alpha, \beta}(0))$ equals $v_p(\beta)$ as desired.

We now proceed with the proof of Case 4.iii. The first two assumptions of Lemma \ref{lem:fxy} are satisfied by our assumptions. Now as $f_{\alpha, \beta}$ is post-critically finite with 0 and 1 being periodic, we have that
\begin{align*}
    f^n_{\alpha, \beta}(0) = 0 \quad \text{and} \quad f^n_{\alpha, \beta}(1) = 1
\end{align*}
for some integer $n$. We then have that
\begin{align*}
    v_p(f^{n-1}_{\alpha, \beta}(f_{\alpha, \beta}(0))) = \infty, \quad & v_p(f^{n-1}_{\alpha, \beta}(f_{\alpha, \beta}(1))) = 0, \\
    v_p(f^{n-1}_{\alpha, \beta}(\beta)) = \infty, \quad\text{and } & v_p(f^{n-1}_{\alpha, \beta}(\alpha B_{d,k}(1) + \beta)) = 0.
\end{align*}
We proceed by setting $X$ to $\beta$ and $Y$ to $\alpha B_{d,k}(1)$ in Lemma ~\ref{lem:fxy}. In order to apply the full power of Lemma ~\ref{lem:fxy}  we must check the remaining two assumptions of Lemma ~\ref{lem:fxy} are satisfied. As $X$ equals $\beta$, the third assumption is satisfied. Finally, the last assumption is that $v_p(Y)$ is greater than $v_p(\alpha)$. We bound $v_p(Y)$ as follows
\begin{align*}
    v_p(Y) &= v_p( \alpha B_{d,k}(1) ) = v_p \left(\alpha \sum_i b_i \right) \\
    &\geq \min \{ v_p(\alpha) + v_p(b_d), \ldots, v_p(\alpha) + v_p(b_{d-k}) \} \geq v_p(\alpha),
\end{align*}
since $v_p(b_i)$ is always non-negative. Therefore, the last assumption is satisfied, and by Lemma ~\ref{lem:fxy} we have that
\begin{align*}
    0 = v_p(f^{n-1}_{\alpha, \beta}(\beta + \alpha B_{d,k}(1))) &= v_p(f^{n-1}_{\alpha, \beta}(\beta) + h_{n-1}(\beta, \alpha B_{d,k}(1))) \\
    & \geq \min \{ v_p(f^{n-1}_{\alpha, \beta}(\beta)), v_p(h_{n-1} (\beta, \alpha B_{d,k}(1))) \} \\
    & \geq \{ \infty, v_p(\alpha) \} \\
    &= v_p(\alpha)
\end{align*}
which is a contradiction as $v_p(\alpha)$ is strictly greater than 0.

\end{proof}

Careful analysis of where the assumptions are used in the proof of Proposition \ref{prop:pintegral} shows that even with weaker assumptions proofs of certain cases still hold. Cases 1 and 4 only rely on 0 being periodic, while Cases 2 and 3 rely only on 1 being periodic. Thus, we get that

\begin{prop}
Let $p$ be an IDF prime for (d,k) and let $f_{a,c}(z) = a\mathcal{B}_{d,k}(z) + c$. If $(\alpha, \beta)$ is a solution to $F_n(a,c) = 0$, then one of $\alpha$ and $\beta$ are $p$-integral. If $(\alpha, \beta)$ is a solution to $G_m(a,c) = 0$ and $\beta$ is $p$-integral, then so is $\alpha$.
\end{prop}
\begin{proof}
If $(\alpha, \beta)$ is a solution to $F_n(a,c) = 0$, then 0 is periodic under $f_{\alpha, \beta}$. Hence, the proofs of Cases 1 and 4 apply, which shows that at least one of $\alpha$ and $\beta$ are $p$-integral. Similarly, if $(\alpha, \beta)$ is a solution to $G_m(a,c) = 0$, then 1 is periodic under $f_{\alpha, \beta}$ and, hence, the proofs of Cases 2 and 3 apply. Thus, if $v_p(\beta)$ is non-negative, we can not be in Cases 1 or 4, and, hence, $\alpha$ must also be $p$-integral.
\end{proof}

Cases 1 and 2 both show the valuation of the orbit of a critical point is unbounded, so that the same proof works even if the critical point was only assumed to be preperiodic. Formally, let
\begin{align*}
    H^{(n_0,m_0)}(a,c) = f_{a,c}^{n_0}(0) - f_{a,c}^{m_0}(0), \quad K^{(n_1, m_1)}(a,c) = f_{a,c}^{n_1}(1) - f_{a,c}^{m_1}(1)
\end{align*}
so that the solutions $(\eta, \gamma)$ to $H^{(n_0,m_0)}(a,c) = K^{(n_1, m_1)}(a,c) = 0$ are such that $f_{\eta, \gamma}$ is PCF with $0$ having preperiod $(n_0, m_0)$ and 1 having preperiod $(n_1, m_1)$.

Then, we have the following proposition.

\begin{prop}\label{prop:preperiodic}
Let $p$ be an IDF prime for (d,k) and let $f_{a,c}(z) = a\mathcal{B}_{d,k}(z) + c$. If $(\eta, \gamma)$ is a solution to $H^{(n_0,m_0)}(a,c) = K^{(n_1, m_1)}(a,c) = 0$, then either both $\eta$ and $\gamma$ are $p$-integral or $v_p(\gamma) < 0$ and $0 < v_p(\eta)$.
\end{prop}

\begin{proof}
Suppose not. Then, we are in one of Cases 1, 2, or 3 from the proof of Proposition \ref{prop:pintegral}. In these cases, we showed that the valuation of the orbit of a critical point under $f_{\eta, \gamma}$ is unbounded. Thus, $f_{\eta, \gamma}$ can not be PCF, because if the orbit of the critical point is finite the valuation of the orbit is necessarily bounded, which is a contradiction.
\end{proof}

Using similar methods as for Proposition \ref{prop:pintegral}, we also prove the following lemma, which will be important in Section \ref{rigidity}.

\begin{lem} \label{lem:v(a)eq0}
Let $p$ be an IDF prime for (d,k) and let $f_{a,c}(z) = a\mathcal{B}_{d,k}(z) + c$. If $(\alpha, \beta)$ is a solution to $F_n(a,c) = G_m(a,c) = 0$, then $\alpha$ is non-zero modulo $p$.
\end{lem}
\begin{proof}
Since the proof for an algebraic extension is not substantially different from the rational case, we will assume $\alpha$ and $\beta$ are rational for simplicity.

We begin by establishing an equation similar to Equation ~\eqref{eq:minleq0} for $x$ with positive valuation.

Let $x$ be an arbitrary element in $\mathbb{Q}(\alpha, \beta)$, and let $v_p$ be a normalized valuation on $\mathbb{Q}(\alpha, \beta)$. Using similar logic as we used for Equation ~\eqref{eq:minleq0}, we find that
\begin{equation}\label{eq:mingeq0}
    v_p(f_{\alpha, \beta}(x)) \geq \min \{ v_p(\alpha) + v_p(d-r) + (d-k)v_p(x), v_p(\alpha) + (d-r)v_p(x), v_p(\beta) \}.
\end{equation}
Now, we consider the following two cases:
\begin{align*}
        \text{\textbf{Case 1. }} &v_p(\alpha) > 0,\;  v_p(\beta) > 0 ;\\
        \text{\textbf{Case 2. }} &v_p(\alpha) > 0,\;  v_p(\beta) = 0.
\end{align*}

\textbf{Case 1.} We show $v_p(f^n_{\alpha, \beta}(1))$ is positive for all positive $n$. We begin by using Equation ~\eqref{eq:mineq0}
\begin{align*}
    v_p(f_{\alpha, \beta}(1)) \geq \min \{v_p(\alpha), v_p(\beta) \} > 0.
\end{align*}
Now, for the inductive step, we can use Equation ~\eqref{eq:mingeq0}
\begin{multline*}
    v_p(f^n_{\alpha, \beta}(1)) \geq \min \{ v_p(\alpha) + v_p(d-r) + (d-k)v_p(f^{n-1}_{\alpha, \beta}(1)),\\
    v_p(\alpha) + (d-r)v_p(f^{n-1}_{\alpha, \beta}(1)), v_p(\beta) \},
\end{multline*}
which is positive as every term is positive.

\textbf{Case 2.} For this case, we show that $v_p(f^n_{\alpha, \beta}(0))$ is zero for all positive $n$. As $f_{\alpha, \beta}(0)$ equals $\beta$, $v_p(f_{\alpha, \beta}(0))$ is zero by assumption. For the inductive step, we can use Equation ~\eqref{eq:mineq0}
\begin{align*}
    v_p(f^n(0)) \geq \min \{ v_p(\alpha), v_p(\beta) \} = 0,
\end{align*}
and the inequality becomes equality since the minimum is unique.

In both Case 1 and Case 2, our conclusions contradicted that $0$ and $1$ are periodic, so are not possible.
Combined with cases 3 and 4 in Proposition ~\ref{prop:pintegral}, we see that $v_p(\alpha)$ must be 0, and our desired result is proved.
\end{proof}

\section{Existence of IDF Primes}\label{idfprimes}

As our proof of $p$-integrality relies on the existence of an IDF prime $p$, we turn our attention to proving when such a prime exists. We begin proving the existence of an IDF prime by considering the cases where $k$ is small.

%Let us define a square-free factor of $n$ to be a prime divisor $p$ such that $p$ divides $n$ but $p^2$ does not. A square-free prime divisor of the product $(d-k)\cdots d$ is an IDF prime provided that it does not divide $d-1$. The above conjecture would imply the existence of a simple prime divisor of $(d-k)\cdots(d-2)$ provided that $k$ is greater than 5. This simple prime divisor will be an IDF prime if $k$ is not prime. The existence of an IDF prime can therefore be viewed as a weakening of Conjecture ~\ref{conj:erdos} in most cases.

\begin{lem}\label{lem:k<4}
For $k \leq 3$, there exists an IDF prime $p$ for $(d,k)$ except when $(d, k)$ equals $(27, 3)$. Further, for any given $k$, there are only finitely many $d$ for which an IDF prime does not exist.
\end{lem}
\begin{proof}
The case where $k$ equals 1 is trivial, so we begin with $k=2$. For this case, the conditions for being an IDF prime $p$ are equivalent to
\begin{itemize}
    \item $p$ is greater than 2
    \item $p$ divides $d$, or $p$ divides $d-2$ and $v_p(d-2)$ is odd.
\end{itemize}
If $d$ has a prime divisor greater than $2$, we are done. Otherwise, $d = 2^i$ for some $i$. Then, we can apply Zsigmondy's theorem \cite{Zsigmondy} to conclude that $d-2 = 2(2^{i-1}-1)$ has a primitive prime divisor when $i$ is greater than $6$. Checking the remaining cases computationally where $i$ is less than 6, we find that there is always an IDF prime.
\begin{comment}
for i in range(2,7):
    d=ZZ(2)^i
    for p in d.prime_divisors():
        if p > 2:
            print(d,p, (d-2).valuation(p))
    for p in (d-2).prime_divisors():
        if p > 2:
            print(d,p, (d-2).valuation(p))
\end{comment}

We now approach the case $k=3$. For this case, the conditions for being an IDF prime are equivalent to
\begin{itemize}
    \item $p$ is greater than 3
    \item $p$ divides d, or $p$ divides $d-2$ and $v_p(d-2)$ is odd, or $p$ divides $d-3$ and $3 \nmid v_p(d-3)$.
\end{itemize}
We proceed by directly computing exceptions. If either $d-3$ or $d-2$ is a multiple of an IDF prime, then we are done. Otherwise, $d- 3$ is not divisible by an IDF prime, so every prime $p$ which is greater than 3 must have $3 \mid v_p(d-3)$. Hence, every prime greater than 3 occurs to a power divisible by 3. This implies that $d - 3 = CX^3$ for some integer $X$, where $C$ is only divisible by $2$ or $3$ at most to the second power (as higher powers of $2$ or $3$ can be absorbed into $X$), so that $C$ is in the set $\{1, \, 2, \,  3, \, 2^2, \, 2\cdot 3, \, 3^2, \, 2^2 \cdot 3, \, 2 \cdot 3^2, \, 2^2 \cdot 3^2\}$.
Similarly, we must have that $d-2 = BY^2$ for some integer $Y$, where $B$ is in the set $\{1, \, 2, \, 3, \, 6\}$. Combining these two equations, we find that integer pairs $(X, Y)$ must satisfy
\begin{align*}
    BY^2 = CX^3 + 1.
\end{align*}

After the substitution
\begin{equation}\label{eq_coords}    X = \frac{x}{BC}\quad \text{and} \quad Y = \frac{y}{B^2C},
\end{equation}
this curve becomes $y^2 = x^3 + B^3C^2$. We find points $(x, y)$ using standard methods to compute integral points on elliptic curves in Sage \cite{sage}, and then compute the points $(X,Y)$ via \eqref{eq_coords}. Note that we have that $k$ must be less than or equal to $\left \lceil \frac{d-2}{2} \right \rceil$, so as $k$ is $3$, $d$ must be at least 7. We find that the only solutions are
\begin{align*}
(X,Y, B, C) \in \{ (2,3, 1,1), (2,5,1,3), (2,7,1,6), (2, 17,1,36), (23, 78,2,1), (61, 389,3,2) \}
\end{align*}
These correspond to the tuples
\begin{align*}
(d, 3) \in \{ (11, 3), (27, 3), (51, 3), (291, 3), (12170, 3), (453965, 3) \}
\end{align*}
For these tuples $(d,3)$, $d$ is a multiple of an IDF prime except when $d = 27$, giving us one exception $(27, 3)$.
\begin{comment}
B = [1,2,3,6]
C = [1,2,3,4,6,9,12,18,36]
R.<x,y>=QQ[]
for b in B:
    for c in C:
        E=EllipticCurve([0,b^3*c^2])
        for P in E.integral_points():
            xP,yP = P[0],P[1]
            if xP*yP != 0:
                X=xP/(b*c)
                Y=yP/(b^2*c)
                if X>0 and int(X) == X and int(Y) == Y and c*X^3 >= 4:
                    print("(X,Y,B,C,d)=",X,Y,b,c,c*X^3+3)
\end{comment}

The above method easily generalizes to any $k$, showing that there are only finitely many possible exceptions $(d,k)$ for any given $k$.
\end{proof}

On the other hand, when $k$ is very large, we can show there is always an IDF prime $p$. Note that in the following lemma, ``effectively computable'' means that the method of proof allows for the explicit computation of $\gamma$.

\begin{lem}\label{lem:k>4}
If $k$ is greater than some effectively computable constant $\gamma$, there is always an IDF prime $p$ for $(d,k)$.
\end{lem}
\begin{proof}
To approach this more general case, we show that there exists a prime $p$ such that $p$ is greater than $k$, $p$ divides $d-r$ for some even $r$ less than $k$, and $v_p(d-r)$ is odd. Consider the product
\begin{align*}
\Delta = d(d-2)(d-4) \cdots \left (d-2 \left \lfloor \frac{k}{2} \right \rfloor \right).
\end{align*}
First we show that there exists a prime $p > k$ which divides $\Delta$. Note that $\Delta$ is the product of an arithmetic sequence with common ratio $2$. Laishram and Shorey \cite{GreatestPrime} prove for arithmetic sequences with common ratio $2$, there exists
a prime $p > 2 \left ( \left \lfloor \frac{k}{2} \right \rfloor + 1 \right) \geq k$ dividing $\Delta$.

Next we show that if $k$ is large enough, there exists a prime $p > k$ such that $v_p(d-r)$ is odd. Assume that for all primes $p >  2 \left ( \left \lfloor \frac{k}{2} \right \rfloor + 1 \right)$ dividing $\Delta$, $v_p(d-r)$ is even. Then, we would have a solution to
\begin{align} \label{eq:by^2}
\Delta = by^2
\end{align}
where $b$ has no prime divisor greater than $2 \left ( \left \lfloor \frac{k}{2} \right \rfloor + 1 \right)$. Filaseta, Laishram, and Saradha \cite{FiniteSolutions} prove that when $k$ is larger than some effectively computable constant, there are no solutions to Equation \eqref{eq:by^2}.
\end{proof}

Combining Lemma ~\ref{lem:k>4} and Lemma ~\ref{lem:k<4}, we have the following proposition.

\begin{prop}
Except for finitely many exceptions $(d,k)$ which are effectively computable, there exists an IDF prime $p$ for $(d,k)$.
\end{prop}

Ideally, we would enumerate all $k$ less than $\gamma$ and compute all possible exceptions $(d,k)$. The constant $\gamma$, however, is quite large, of the order of $50^{50}$. It is therefore infeasible to compute all exceptions.

Additionally, we note that the method used in Lemma ~\ref{lem:k<4} to compute the possible exceptions $(d,k)$ requires computing integral points on $6^{\pi(k)}$ elliptic curves, where $\pi(k)$ is the number of primes less than or equal to $k$. Thus, the authors were only able to compute exceptions up until $k = 10$. We found that for all pairs $(d,k)$ with $k \leq 10$ except for $(27,3)$, there exists an IDF prime.

A more efficient method would be as follows. If $d$ is not a multiple of an IDF prime, it is the product of primes less than or equal to $k$. If $d-2$ is not a multiple of an IDF prime, then it is of the form $BX^2$, where $B$ is in the set described in Lemma ~\ref{lem:k<4}. We thus get an equation of the form
\begin{align*}
    BX^2 + 2 = p_1^{z_1} \cdots p_n^{z_n}.
\end{align*}
Equations of this form are solved in \cite{PrimePowerDiophantine}. For any given $k$, there will be $\mathcal{O}(\pi(k)^2)$ possibilities for $B$, giving a much more efficient method for computing exceptions.

As computing all exceptions is not currently feasible, we instead state Conjecture \ref{conj:goodprime}. In the spirit of Conjecture ~\ref{conj:erdos}, we rephrase Conjecture ~\ref{conj:goodprime} to have a more number theoretic flavor.

\begin{conj}
Let $k$ be a non-negative integer and let $n$ be greater than $2k+2$. There exists a prime $p$ greater than $k$ which divides $\Delta = n(n-1)\cdots (n-k)$. Moreover, the index at which $p$ occurs in the product $n(n-1)\cdots (n-k)$ does not divide the power to which it occurs in $\Delta$, except when $(n,k)$ equals $(27, 3)$.
\end{conj}

\section{Thurston Rigidity}\label{rigidity}

Having proved that the PCF solutions are $p$-adically integral, the next step towards proving the main theorem is to prove that the curves $F_n(a,c)$ and $G_m(a,c)$ intersect transversely. In order to show that $F_n$ and $G_m(a,c)$ intersect transversely, we consider the Jacobian
\begin{align*}
J(a,c) = \det \begin{pmatrix}
    (F_n)_a  & (G_m)_a \\
    (F_n)_c  & (G_m)_c
\end{pmatrix} \in \mathbb{Z}[a,c]
\end{align*}
where the subscript indicates a partial derivative. Then the curves $F_n = 0$ and $G_m = 0$ intersect transversely at all their points of intersection if and only if the ideal
\begin{align*}
    (F_n, G_m, J) \in \mathbb{C}[a,c]
\end{align*}
is the unit ideal.  We prove that $(F_n, G_m, J) = (1)$ by proving $J(a,c)$ does not vanish modulo $p$ when $p$ is an IDF prime.

\begin{prop} \label{prop:Jacobianneq0}
Let $p$ be an IDF prime for $(d,k)$, and let $(\alpha, \beta)$ be a solution to $F_n(a,c) = G_m(a,c) = 0$. Then, the Jacobian \begin{align*}
    J(a, c) = \det \begin{pmatrix}
    (F_n)_a & (G_m)_a \\
    (F_n)_c & (G_m)_c
\end{pmatrix}
\end{align*}
is non-zero modulo $p$ when $(a,c)$ equals $(\alpha, \beta)$.
\end{prop}

\begin{proof}
Since the proof for an algebraic extension is not substantially different from the rational case, we will assume $\alpha$ and $\beta$ are rational for simplicity.

Begin by choosing an IDF prime ideal $p$ in $\mathbb{Q}(\alpha, \beta)$, and let $v_p$ be the valuation normalized with respect to this prime ideal. Letting $b_0, \ldots, b_k$ be the coefficients of $B_{d,k}(z)$, we know that $v_p(b_i)$ equals $v_p(d-r)$, except when $i=d-r$, as $v_p(b_{d-r})$ equals $0$. As $v_p(d-r)$ is greater than 0, we have that $\mathcal{B}_{d,k}$ reduces to a monomial modulo $p$ by reducing the coefficients.

Now we can reduce $f_{a, c}$ using the reduction of $\mathcal{B}_{d,k}$:
\begin{equation*}
\overline{f}_{a, c}(z) \equiv as z^{t p}+ c  \;(\bmod\; p),
\end{equation*}
where
\begin{equation*}
s \equiv  (-1)^{k-r}\prod_{j = 0, j \neq r}^k (d - j)\cdot\frac{1}{(k - r)!r!} \;(\bmod\; p)
\end{equation*}
and $t p = d-r$ for some $t \in \mathbb{N}$. Since the critical points of $f$ in this bicritical normal form are 0 and 1, given periods $m,n \geq 1$, the intersection of
\begin{equation*}
    F_n(a,c) = f^n_{a,c}(0) = 0 \quad \text{and} \quad   G_m(a,c) = f^m_{a,c}(1) - 1 = 0
\end{equation*}
gives the locus of PCF bicritical polynomials with 0 periodic of period $n$ and 1 periodic of period $m$. We compute the Jacobian of these two curves and show that it cannot be 0 mod $p$ at the points of intersection. We can explicitly compute the partial derivatives of $\overline{f}^{n}_{a,c}(0)$ and
$\overline{f}^{m}_{a,c}(1)$ as follows:
\begin{align*}
 \frac{\partial}{\partial a}(\overline{f}^n_{a,c}(0)) &\equiv \frac{\partial}{\partial a}\left(as (\overline{f}^{n-1}_{a,c}(0))^{t p} + c\right) \\
&\equiv s (\overline{f}^{n-1}_{a,c}(0))^{t p} + as t p (\overline{f}^{n-1}_{a,c}(0))^{t p - 1}\frac{\partial}{\partial a}\left(\overline{f}_{a,c}^{m-1}(0) \right) \\
&\equiv s (\overline{f}^{n-1}_{a,c}(0))^{t p} \pmod{p}\\\\
    \frac{\partial}{\partial c} \left(\overline{f}^n_{a,c}(0)\right) &\equiv \frac{\partial}{\partial c} \left(as (\overline{f}^{n-1}_{a,c}(0))^{t p} + c\right)\\
&\equiv a s t p (\overline{f}^{n-1}_{a,c}(0))^{t p - 1}\frac{\partial}{\partial c}\left(\overline{f}^{n-1}_{a,c}(0) \right) + 1\\
&\equiv 1 \pmod{p}.
\end{align*}
Similarly,
\begin{align*}
\frac{\partial}{\partial a} \left(\overline{f}^m_{a,c}(1)\right) &\equiv s (\overline{f}^{m-1}_{a,c}(1))^{t p} \pmod{p}\\
\frac{\partial}{\partial c} \left(\overline{f}^m_{a,c}(1)\right) &\equiv 1 \pmod{p}.
\end{align*}
Thus, the Jacobian is given by
\begin{align*}
    J(a,c) &\equiv \det
    \begin{pmatrix}
        1 & 1 \\
        s (\overline{f}^{n-1}_{a,c}(0))^{t p} & s (\overline{f}^{n-1}_{a,c}(1))^{t p}\\
    \end{pmatrix} \\
    &\equiv s(\overline{f}^{m-1}_{a,c}(1)^{t p} - \overline{f}^{m-1}_{a,c}(0)^{t p}) \pmod{p}.
\end{align*}

Now we evaluate at a point of intersection $(\alpha, \beta)$, which is a solution to the equations $F_n(a,c) = G_m(a,c) = 0$. Denote the reductions of $\alpha$ and $\beta$ modulo $p$ by $\overline{\alpha}$ and $\overline{\beta}$. By Proposition ~\ref{prop:pintegral} both $\overline{\alpha}$ and $\overline{\beta}$ are defined, and by Lemma ~\ref{lem:v(a)eq0} $\overline{\alpha}$ is non-zero. Since $f^m_{\alpha, \beta}(1) - 1$ equals $0$, we have that $\overline{\alpha} s(f^{m-1}_{\alpha,\beta}(1))^{t p} + \overline{\beta}$ must be equivalent to $1$, so that
\begin{equation*}
    (f^{m-1}_{\alpha, \beta}(1))^{t p} = \frac{1-\overline{\beta}}{\overline{\alpha} s}
\end{equation*}
and, similarly,
\begin{equation*}
    (f^{n-1}_{\alpha, \beta}(0))^{t p} = -\frac{\overline{\beta}}{\overline{\alpha} s}.
\end{equation*}
It follows that
\begin{align*}
    J(\alpha, \beta) &\equiv s(f^{m-1}_{\alpha,\beta}(1)^{t p} - f^{n-1}_{\alpha, \beta}(0)^{t p})\\
    &\equiv s \left(\frac{1-\overline{\beta}}{\overline{\alpha} s} + \frac{\overline{\beta}}{\overline{\alpha} s}\right)\\
    &\equiv \frac{1}{\overline{\alpha}} \pmod{p}.
\end{align*}
As $\overline{\alpha}$ is non-zero by Lemma \ref{lem:v(a)eq0}, $J(\alpha, \beta)$ is defined, and
$$
J(\alpha, \beta) \not\equiv 0 \pmod{p}.
$$
\end{proof}

We are now ready to prove the main theorem. Our proof is identical to the one given in \cite{Silverman-transversality}, however we reproduce it here for the sake of completeness.

\begin{proof} (Main Theorem)
Begin by considering the Jacobian $J(a,c)$ evaluated at a point of intersection $(\alpha, \beta)$ of $F_n(a,c) = G_m(a,c) = 0$. By Proposition ~\ref{prop:Jacobianneq0}, there exists some number $K(\alpha, \beta)$ such that
\begin{equation*}
    J(\alpha, \beta) = \frac{1}{\overline{\alpha}} + p K(\alpha, \beta).
\end{equation*}
By taking norms down to $\mathbb{Q}$, we find that
\begin{align*}
    \textbf{N}_{\mathbb{Q}(\alpha, \beta) / \mathbb{Q}} J(\alpha, \beta) = \textbf{N}_{\mathbb{Q}(\alpha, \beta) / \mathbb{Q}} \left(\frac{1}{\overline{\alpha}} + p K(\alpha, \beta)\right) \equiv \frac{1}{\overline{\alpha}} \pmod{p}.
\end{align*}
In particular, $J(\alpha, \beta)$ is non-zero. It follows that
\begin{align*}
    (F_n, G_m, J) \subset \mathbb{C}[a,c]
\end{align*}
is the unit ideal, since if it were not, the curves $F_n$, $G_m$, and $J$ would have a common root.
\end{proof}

\section{Failure to Extend to $n$ critical points}\label{n-critical-points}

One possible way to extend the above results would be to construct a normal form similar to the Belyi normal form but for more critical points, and try to push through similar proofs. The following shows, however, that the natural generalization to $n$ critical points does not allow for proofs following the standard mode of attack, as the Jacobian is not non-zero modulo $p$.

In order to attempt to generalize to more critical points, we first need a normal form. We give a generalization to $n$-critical points of the normal form for bicritical maps given in Tobin \cite{Bella}. First, we define some notation.
\begin{defn}
We notate the $n+1$ nested sums as
\begin{align*}
    \sum_{j_0\ldots j_n = 0}^{k_0\ldots k_n} \vcentcolon = \sum_{j_0= 0}^{k_0}\sum_{j_1=0}^{k_1}\ldots \sum_{j_n=0}^{k_n}
\end{align*}
\end{defn}
For the following proposition, any sum with lower and upper bounds omitted is a sum from $0$ to $n-2$. That is,
\begin{align*}
    \sum_{l} \vcentcolon = \sum_{l = 0}^{n-2}
\end{align*}

\begin{thm}
\label{n-critical normal form}
Let g $\in K[z]$ be a degree $d$ polynomial with $n$ critical points. There exist
$k_0, \ldots k_{n-2} \in \mathbb{N}$, with $n-1 \leq \displaystyle{\sum_l} k_l \leq d-2$ %this inequality might be irrelevant. also, check this inequality
and $a,c, \gamma_0 \ldots \gamma_{n-2} \in \overline{K}$ such that $g$ is conjugate to
$$
a \left( \frac{d!}{(d - \sum_l k_l - 1)!} \sum_{j_0 \ldots j_{n-2} = 0} ^ {k_0 \ldots k_{n-2}} \left( \prod_{i = 0}^{n-2} (-\gamma_i)^{k_i-j_i}  \binom{k_i}{j_i} \right) \frac{z^{d+\sum_l j_l - k_l}}{d + \sum_l j_l-k_l} \right) + c.
$$
\end{thm}
\begin{proof}
Let $g \in K[z]$ be a polynomial with $n$ critical points $\xi_0, \ldots, \xi_{n-1}$. Let $\phi(z) = \frac{z-\xi_{n-1}}{\xi_{0}-\xi_{n-1}} \in \text{PGL}_2(\overline{K})$, which sends the critical points $\xi_0$ and $\xi_{n-1}$ to
1 and 0, respectively. Then, $f(z) = g^{\phi^{-1}}$ has critical points $0, \gamma_0 = 1, \gamma_1, \ldots, \gamma_{n-2}$. Let $d - \sum_ik_i$ be
the ramification index of $0$ and let $k_0+1, \ldots, k_{n-2}+1$ be the ramification indices of $\gamma_0, \ldots, \gamma_{n-2}$, respectively. Then there exists $\alpha \in \overline{K}$ such that
\begin{equation*}
    f'(z) = \alpha z^{d-\sum_l k_l-1}\prod_{i=0}^{n-2}(z-\gamma_i)^{k_i}
\end{equation*}
\begin{align*}
    f(z)
    %&= \int \alpha z^{d-\sum_ik_i-1}\prod_{i=0}^{n-2}(z-\gamma_i)^{k_i} \, dz
    %&= \alpha \int z^{d-\sum_i k_i-1}\prod_{i=0}^{n-2}\sum_{j_i=0}^{k_i}(-\gamma_i)^{k_i-j_i}\binom{k_i}{j_i}z^{j_i} \, dz\\ %this one seems wrong (up to constant multiple)
    &= \alpha \int z^{d-\sum_l k_l -1}\sum_{j_0\ldots j_{n-2} = 0}^{k_0\ldots k_{n-2}} \left(\prod_{i=0}^{n-2} (-\gamma_i)^{k_i-j_i}\binom{k_i}{j_i}z^{j_i} \right) \, dz\\
    &= \alpha \int z^{d-\sum_l k_l-1}\sum_{j_0\ldots j_{n-2} = 0}^{k_0\ldots k_{n-2}}z^{\sum_l j_l}\left[\prod_{i=0}^{n-2}(-\gamma_i)^{k_i-j_i}\binom{k_i}{j_i}\right] \, dz\\
    &= \alpha \sum_{j_0\ldots j_{n-2} = 0}^{k_0\ldots k_{n-2}} \left[\prod_{i=0}^{n-2}(-\gamma_i)^{k_i-j_i}\binom{k_i}{j_i} \right] \int z^{d+\sum_l j_l-k_l-1} \, dz\\
    &= \alpha \left( \sum_{j_0\ldots j_{n-2} = 0}^{k_0\ldots k_{n-2}} \left[\prod_{i=0}^{n-2}(-\gamma_i)^{k_i-j_i}\binom{k_i}{j_i} \right]  \frac{z^{d+\sum_l j_l-k_l}}{d+\sum_l j_l-k_l} \right)+ c.
\end{align*}
We can then make the substitution
\begin{equation*}
    \alpha = a \left(\frac{d!}{(d - \sum_l k_l - 1)!} \right)
\end{equation*}
so that
\begin{equation*}
    f(z) = a \left(\frac{d!}{(d - \sum_l k_l - 1)!} \sum_{j_0\ldots j_{n-2} = 0}^{k_0\ldots k_{n-2}} \left[ \prod_{i=0}^{n-2}(-\gamma_i)^{k_i-j_i}\binom{k_i}{j_i} \right]  \frac{z^{d+\sum_l j_l-k_l}}{d+\sum_l j_l-k_l}\right) + c.
\end{equation*}
Additionally, since $k_0+1, \ldots k_{n-2}+1$ are the ramification indices of the critical points, for all $l$ we have $k_l + 1 \geq 2$, and hence $k_l \geq 1$.
Since $d - \sum_l k_l$ is the ramification index of 0,
%$$d - \sum_i k_i \geq 2$$
\begin{equation*}
    \sum_l k_l \leq d-2 .
\end{equation*}
Combining these two inequalities, we have
\begin{equation*}
    n-1 \leq \sum_l k_l \leq d-2. \qedhere
\end{equation*}
\end{proof}

Applying Theorem \ref{n-critical normal form} to a particular case leads to a specific normal form as demonstrated in the next example.
\begin{ex}
\label{n-critical normal form example 1}
%\textcolor{blue}{Consider the family of degree $4$ polynomials with three critical points, $\gamma_0'$, $\gamma_1'$, and $\gamma_2'$ all of ramification index $2$. Theorem \ref{n-critical normal form} states that this family is conjugate to a family $f_{a,c,\gamma}$ of the form given in Theorem \ref{n-critical normal form}. Note that the proof of Theorem \ref{n-critical normal form} conjugates so that $\gamma_0'$ is sent to $\gamma_0 = 1$, $\gamma_1'$ is sent to $\gamma_1$, and $\gamma_2'$ is sent to $0$. The critical points of $g_{a,c,\gamma}$ are thus $0$, $1$, and $\gamma_1$, which we relabel as $\gamma$ for clarity.}

Consider a degree 4 polynomial with 3 critical points, $\gamma_0, \gamma_1, \gamma_2$, each of ramification index 2.
%Let $f_{a,c,\gamma}(z)$ be of degree 4 with 3 critical points, $\gamma_0, \gamma_1, \gamma_2$, each of ramification index 2.
Note that the proof of Theorem \ref{n-critical normal form} shows that we can conjugate so that $\gamma_0 = 1$ and $\gamma_2 = 0$. For clarity, we relabel $\gamma_1 = \gamma$. Label the normal form given for this polynomial from Theorem \ref{n-critical normal form} as $f_{a,c,\gamma}(z)$. The proof of Theorem
\ref{n-critical normal form} also shows that $k_i = e_{\gamma_i}(f_{a,c,\gamma}) - 1$, where $e_{\gamma_i}(f_{a,c,\gamma})$ is the ramification index of $f_{a,c,\gamma}$ at
$\gamma_i$. Hence, we have that $k_0 = e_{\gamma_0}(f_{a,c,\gamma}) - 1 = 1$ and $k_1 = e_{\gamma_1}(f_{a,c,\gamma}) - 1 = 1$. Now, we substitute with $d = 4$, $n = 3$, $k_0 = 1$, $k_1 = 1$, $\gamma_0 = 1$, and $\gamma_1 = \gamma$.
\begin{align*}
    f_{a,c,\gamma}(z) = a \left( \frac{4!}{4-\sum_l k_l-1!} \sum_{j_0, j_1 = 0} ^ {1, 1} \left[ \prod_{i = 0}^{1} (-\gamma_i)^{k_i-j_i} \binom{k_i}{j_i} \right] \frac{z^{4+\sum_{l=0}^{1}j_l - \sum_{l=0}^1 k_l}}{4 + \sum_{l = 0}^{1}j_l - \sum_{l = 0}^1 k_l} \right) + c.
\end{align*}
Since $k_i = 1$ for all $i$, and $\sum_l k_l = 2$, it follows that
\begin{align*}
    f_{a,c,\gamma}(z) = a \left( 4! \sum_{j_0, j_1 = 0} ^ {1, 1} \left[ \prod_{i = 0}^{1} (-\gamma_i)^{1-j_i} \binom{1}{j_i} \right] \frac{z^{2+j_0+j_1}}{2 + j_0 + j_1} \right) + c.
\end{align*}
We can bring the $4!$ into the sum and use the fact that $j_i \leq k_i = 1$ to simplify $\frac{4!}{2+j_0+j_1}$ as $\prod_{l = 0, \, l \neq j_0+j_1}^2 2+l$:
\begin{align*}
    f_{a,c,\gamma} &= a \left(\sum_{j_0, j_1 = 0} ^ {1, 1} (-1)^{1-j_0}(-\gamma)^{1-j_1} \left[\prod_{l = 0, \, l \neq j_0+j_1} ^{2} 2+l\right]z^{2+j_0+j_1} \right) + c\\
    &= a \sum_{j_0 = 0}^{1} \left((-1)^{1-j_0}(-\gamma)\left[\prod_{l = 0, \, l \neq j_0}^{2} 2+l\right]z^{2+j_0} + (-1)^{1-j_0}\left[\prod_{l = 0, \, l \neq j_0+1}^{2} 2+l\right]z^{3+j_0} \right) + c \\
    &= a\left(\gamma(3\cdot 4)z^2-(2\cdot 4)z^3-\gamma(2\cdot 4)x^3+(2\cdot 3)z^4 \right)\\
    &= a(6z^4 - 8(1+\gamma)z^3 + 12\gamma z^2)+c.
\end{align*}
\end{ex}

Using the normal form in Theorem \ref{n-critical normal form}, we might hope to provide algebraic proofs of
transversality for polynomials with 3 or more critical points. Unfortunately, the next two examples show that choosing $p$ as in the bicritical case ultimately fails.

Example \ref{exmp_1} shows the importance of the results on conjugacy from \cite{Bella} in the bicritical case. For polynomials with 3 or more critical points, we can no longer assume that $1 \leq k_i \leq \lceil \frac{d-2}{2} \rceil$, which means
%I can prove that 1 <= k_{n-2} <= ceiling((d-2)/2), but its not particularly interesting. it can be done by relabeling the critical points - Alex
we cannot always find a prime for which the polynomial reduces nicely.

\begin{ex} \label{exmp_1}
Consider the family of degree $10$ polynomials with three critical points, $\gamma_0'$, $\gamma_1'$, and $\gamma_2'$ of ramification indices $8$, $2$, and $2$ respectively. Theorem \ref{n-critical normal form} states that we can parametrize this family by polynomials $g_{a,c,\gamma}$ of the form given in Theorem \ref{n-critical normal form}. Note that the proof
of Theorem \ref{n-critical normal form} conjugates so that $\gamma_0'$ is sent to $1$, $\gamma_1'$ is sent to $\gamma$, and $\gamma_2'$ is sent to $0$. The critical points of $g_{a,c,\gamma}$ are thus $0$, $1$, and
$\gamma$. Also note that $k_i = e_{\gamma_i}(g_{a,c,\gamma})-1$,
hence $k_0 = 7$ and $k_1 = 1$. Substituting $d=10$, $n=3$, $k_0 = 7$, $k_1 = 1$, $\gamma_0 = 1$, and $\gamma_1 = \gamma$, we have
\begin{align*}
    g_{a,c,\gamma}(z) &= a \left( \frac{10!}{1!} \sum_{j_0, j_1 = 0} ^ {7, 1} \left[\prod_{i = 0}^{1} (-\gamma_i)^{k_i-j_i} \binom{k_i}{j_i} \right] \frac{z^{2+j_0+j_1}}{2 + j_0+j_1} \right) + c\\
    &= a \left(\sum_{j_0, j_1 = 0} ^ {7, 1} \left[\prod_{i = 0}^{1} (-\gamma_i)^{k_i-j_i}\binom{k_i}{j_i} \right]\left(\prod_{l = 0, l \neq j_0+j_1} ^{8} 2+l\right)z^{2+j_0+j_1} \right) + c.
\end{align*}
Note that the product
\begin{equation*}
    \prod_{l = 0, l \neq j_0+j_1} ^{8} 2+l
\end{equation*}
will always be nonzero modulo any prime $p > 10$ and will always be 0 modulo any prime $p \leq 5$. We must then try to reduce
$g_{a,c,\gamma}$ modulo $p = 7$. Values of $j_0$ and $j_1$ that will make the previous product nonzero modulo 7 are those such that
$j_0+j_1 = 5$, which results in two solutions, $j_0 = 5$ and $j_1 = 0$ or $j_0 = 4$ and $j_1 = 1$. In both of these cases,
however, we have
\begin{equation*}
    7 \mid \binom{k_0}{j_0}
\end{equation*}
as $\binom{k_0}{j_0} = \binom{7}{5}$ or $\binom{7}{4}$, giving
\begin{equation*}
    g_{a,c,\gamma} \equiv c \pmod{7}.
\end{equation*}
Hence there is no prime for which the reduction $\overline{g}_{a,c,\gamma}$ is useful for proving transversality.
\end{ex}
If we assume that $1 \leq k_i \leq \left\lceil \frac{d-2}{2} \right\rceil$, then we can apply Sylvester's
 theorem \cite{Sylvester} and the Bertrand-Chebyshev \cite{Chebyshev} theorem to guarantee a prime for which the reduction is sufficiently nice. However, as Example \ref{exmp_2} shows, this is not enough to be able to prove transversality.

\begin{ex}\label{exmp_2}
Consider the family $f_{a,c,\gamma}$ from \hyperref[n-critical normal form example 1]{Example \ref*{n-critical normal form example 1}}.
\begin{align*}
    f_{a,c,\gamma}(z) &= a(6z^4 - 8(1+\gamma)z^3 + 12\gamma z^2)+c.
\end{align*}
Clearly, we must reduce by $p = 3$ to get
\begin{equation*}
    f_{a,c,\gamma} \equiv a(1+\gamma)x^3+c \pmod{3}.
\end{equation*}
We compute the Jacobian $J(a,c,\gamma)$ as
\begin{align*}
    J(a,c,\gamma) &= \det \begin{pmatrix}
1 & 1 & 1\\
(1+\gamma)(f^{m-1}(0))^3 & (1+\gamma)(f^{n-1}(1))^3 & (1+\gamma)(f^{k-1}(\gamma))^3\\
a(f^{m-1}(0))^3 & a(f^{n-1}(1))^3 & a(f^{k-1}(\gamma))^3 \\
\end{pmatrix}\\
&= a(1+\gamma)\det \begin{pmatrix}
1 & 1 & 1\\
(f^{m-1}(0))^3 & (f^{n-1}(1))^3 & (f^{k-1}(\gamma))^3\\
(f^{m-1}(0))^3 & (f^{n-1}(1))^3 & (f^{k-1}(\gamma))^3 \\
\end{pmatrix},
\end{align*}
which is zero as the second and third rows are equal.
\end{ex}

As the above example shows, we can not prove transversality algebraically with this normal form. We do, however, wonder if the PCF solutions for the $n$-critical normal form are $p$-adically integral.

\bibliographystyle{plain}

\end{document}